\documentclass[abstracton]{scrartcl}
\usepackage[normalem]{ulem}
\usepackage{amsfonts}
\usepackage{amsmath}
\usepackage{amssymb}
\usepackage{xcolor}
\usepackage{dsfont}
\usepackage{epsf}
\usepackage{epstopdf}
\usepackage{graphics}
\usepackage{graphicx}
\usepackage{latexsym}
\usepackage{oldgerm}
\usepackage{dsfont}
\usepackage{amscd, amsthm,enumerate,verbatim,calc}
\usepackage{authblk}

\newcommand {\id}   {\ensuremath{\text{\rm id}}}

\newcommand{\Aut}{\ensuremath{\operatorname{Aut}}}

\newtheorem{theorem}{Theorem}%[section]
\newtheorem{lemma}[theorem]{Lemma}

\newtheorem{corollary}[theorem]{Corollary}
\newtheorem{conjecture}[theorem]{Conjecture}

\newtheorem{problem}[theorem]{Problem}

\usepackage[colorlinks,breaklinks,backref]{hyperref}
\usepackage{hyperref}
\usepackage{backref}

\begin{document}

\title
{Distinguishing finite and infinite trees of arbitrary cardinality\\
\bigskip
}

\author[1,2]{Wilfried Imrich}
\author[2]{Rafa{\l}  Kalinowski}
\author[3]{Florian Lehner}
\author[2]{Monika Pil\'sniak}
\author[2]{Marcin Stawiski}
\affil[1]{\normalsize Montanuniversit\"at Leoben, A-8700 Leoben, Austria}
\affil[2]{\normalsize AGH University of Krakow, 30-059 Krakow, Poland}
\affil[3]{\normalsize University of Auckland, Auckland, New Zealand}

\date{\today}
\maketitle
\begin{abstract}
Let $G$ be a finite or infinite graph and $m(G)$  the minimum number of vertices moved by the non-identity automorphisms of $G$.   We are interested in bounds  on the supremum  $\Delta(G)$ of the degrees of the vertices of $G$ that assure the existence of vertex colorings of $G$ with two colors  that are preserved only by the identity automorphism, and, in particular, in the number $a(G)$ of such colorings that are mutually inequivalent.

For trees $T$ with finite
 $m(T)$ we obtain the bound $\Delta(T)\leq2^{m(T)/2}$ for the existence of such a coloring, and show that $a(T)= 2^{|T|}$ if $T$ is infinite. Similarly,  we prove that $a(G) = 2^{|G|}$ for all  tree-like graphs $G$ with $\Delta(G)\le 2^{\aleph_0}$.

For rayless or one-ended trees $T$ with arbitrarily large infinite  $m(T)$,  we prove directly that $a(T)= 2^{|T|}$ if  $\Delta(T)\le 2^{m(T)}$.

\end{abstract}

\noindent
{Keywords: Automorphisms, finite and infinite trees, distinguishing and asymmetrizing number.\\
Math. Subj. Class.: {05C05, 05C15, 05C25, 05C63.} }

\section{Introduction}\label{sec:intro}
Given a graph $G$, a set $S \subseteq V(G)$ is  \emph{distinguishing} if the identity is the only element of $\Aut(G)$ that fixes $S$ setwise. A graph with a distinguishing set is  called \emph{$2$-distinguishable}, and the number $a(G)$ of inequivalent distinguishing sets of $G$ is its \emph{asymmetrizing number}.
We investigate the dependence of $a(G)$ on bounds on the supremum  $\Delta(G)$ of the degrees of the vertices of $G$, and in terms of the \emph{motion} $m(G)$ of $G$, which is  the minimum number of vertices moved by the non-identity automorphisms of $G$.

Our main results are Theorems \ref{thm:fm} and \ref{thm:aleph}. In Theorem \ref{thm:fm} we show that  a tree $T$ with finite motion $m(T)$ and  $\Delta(T) \le 2^{m/2}$  is 2-distinguishable, and that its asymmetrizing number $a(T)$ is $2^{|T|}$ if $T$ is infinite. %The  part of this theorem pertaining to 2-distinguishability, improves a bound of Imrich and Tucker {\cite[Theorem 1]{imtu-2020}.

In Theorem \ref{thm:aleph} we consider tree-like graphs $G$,  and show that $a(G)= 2^{|G|}$ if $\Delta(G)\le 2^{\aleph_0}$.
For its proof  we invoke Theorem \ref{thm:main}, which is a consequence of a theorem of Polat {\cite[Theorem 3.3]{po-1992}. The latter theorem is the main result in a series of three papers by Polat and Sabidussi \cite{posa-1991}, and Polat \cite{po-1991,po-1992}.

Since the proof of {\cite[Theorem 3.3]{po-1992} is somewhat difficult to read, it would be good to have simpler proofs both for this theorem, and for Theorem \ref{thm:main}. In the last section we provide such a proof for the special case of Theorem \ref{thm:main} when $T$ is rayless or one-ended.  The proof uses properties of the  Schmidt rank of rayless graphs.

A curious minor result is Lemma \ref{le:cameron}, which establishes a relationship between the order of a connected graph $G$, the size of $\Aut(G)$, and $a(G)$. In Sections \ref{sec:double} and \ref{sec:treelike} we  also pose  a conjecture and a problem.

Papers on 2-distinguishability rarely explicitly use the number of distinguishing sets, although it can help in proving 2-distinguishability. An exception is Babai's groundbraking 1977 paper \cite{ba-1977}, in which he showed that each tree $T$,  all vertices of which have the same degree $\alpha\ge 2$ of arbitrarily large cardinality, is ``asymmetrizable by $2^{|T|}$ inequivalent asymmetric 2-colorings". In the current  terminology,  this means that  $T$ is 2-distinguishable by $2^{|T|}$ inequivalent distinguishing sets.

Later Babai's result was reproved and generalised  by entirely different means by Polat and Sabidussi \cite{posa-1991}, and Polat \cite{po-1991,po-1992}. They also introduced the notation $a(T)$ for the number of inequivalent distinguishing sets.

Apart from the concept of motion and the asymmetrizing number,  our terminology is  that of Albertson and Collins \cite{ac-1996}. They called a graph \emph{$d$-distinguishable} if there is a vertex coloring with $d$ colors,  such that no non-identity  automorphism preserves the coloring. Hence, when a graph is $2$-distinguishable, then the sets of either color are distinguishing.

 We   focus on the connection between 2-distinguishability  and the  motion of a graph. Such a relationship was already noted by Cameron, Neumann  and Saxl  \cite{canesa-1984} in 1984, where they studied the interdependence between the existence of distinguishing sets of permutation groups and their degrees, where the degree of a permutation group corresponds to our concept of motion.  Another  result of this type is from
 Russel and Sundaram \cite{rusu-1998}, who showed in 1998 that  $G$ is $d$-distinguishable if $d^{m(G)/2} > |\Aut(G)|$.
For further results on motion and distinguishabiliy we refer to the paper of  Conder and Tucker \cite{cotu-2011}, which also contains an extensive bibliography.

The most substantial contribution linking motion to asymmetrization  is due to Babai \cite{ba-2021}. In 2022 he proved that each connected, locally finite graph with infinite motion is 2-distinguishable, thereby verifying the Infinite Motion Conjecture of Tucker \cite{tu-2011}.

\section{Preliminaries}\label{sec:prelim}

We call a graph $G$  2-distinguishable if it has a distinguishing set of vertices, that is, a set $S \subseteq V(G)$ that is preserved only by the identity automorphism. If $S$ is such a set, then its complement $V(G)\setminus S$ is also distinguishing.
The definition allows that $S$ or $V(G)\setminus S$ are empty. In this case the identity is the only automorphism of $G$, and we call $G$ \emph{asymmetric}.  Note that pincking a subset $S$ of $V(G)$ is equivalent to coloring the vertices in $S$ and in $V(G)\setminus S$ with two distinct colors. If $S$ is a distinguishing set, then we call the resulting coloring a \emph{distinguishing coloring.}

Subsets $S$ of $G$ and $S'$ of $G'$ are called \emph{equivalent} if there exists an isomorphism $\varphi$ from $G$ to $G'$ such that $\varphi(S) = S'$. Recall that we defined the {asymmetrizing number  $a(G)$ of $G$, as the number of pairwise inequivalent distinguishing sets of $G$. Clearly $a(G) \leq 2^{|G|}$ for all graphs,  $a(G) = 2^{|G|}$ for asymmetric graphs, and $a(G)=0$ when $G$ is not 2-distinguishable. For example, $a(K_3)= 0$, % because $K_3$ is not 2-distinguishable,
  $a(K_2)= 1$, and $a(K_1)=2$; the inequivalent distinguishing sets of $K_1$ being  $V(K_1)$ and the empty set.

We shall write $(T,w)$ for a tree $T$ with root $w \in V(T)$, {and  $\Aut(T,w)$ for  the subgroup of $\Aut(T)$ that fixes $w$. We say a subset $S \subseteq V(T)$ \emph{distinguishes}  $(T,w)$ if the identity is the only element of $\Aut(T,w)$ that fixes $S$ setwise. Two subsets $S$ and $S'$ are called \emph{equivalent with respect to $\Aut(T,w)$}, if there is an element of $\Aut(T,w)$ which maps $S$ to $S'$,  and we define $a(T,w)$ as the number of inequivalent distinguishing sets of $(T,w)$. Note that  $a(T) \le a(T,w)$, because $\Aut(T) \supseteq \Aut(T,w)$.%, and $a(T) = 2^{|T|}$ if $T$  is asymmetric. %For example, $a(K_{1,3})=0$, but $a(K_{1,3},w)=4$ for any leaf $w$ of $K_{1,3}$.

%Let $(T,w)$ be a rooted {tree}.
For vertices $x$, $y$ of a rooted tree $(T,w)$,  $x$ is  a \emph{descendant} of $y$, in symbols, $x \ge y$,  if $y$ lies on the unique path from $w$ to $x$. Vertex $y$ is the \emph{parent} of $x$ if $y\le x$ and $yx$ is an edge. In that case,  $x$ is a \emph{child} of $y$. We say $x$ and $x'$ are \emph{siblings} if they have the same parent, and  that $x$ and $x'$ are \emph{twins} if they are siblings, and if there is an automorphism that moves $x$ to $x'$ (and fixes their parent). We call the set of twins of $x$ the \emph{similarity class} of $x$, denote it by $\bar{x}$, and call its size %$|\bar{x}|$
the multiplicity $\mu(x)$ of $x$.  We always have $\mu(x) \geq 1$, because $x$ is a twin of itself. By $R_w$ we denote a set of representatives for the similarity classes of the children of the root $w$.

For a vertex  $x$ of $(T,w)$ with parent $y$,  let $T^x$ denote the component of $T - xy$ which contains $x$. For $x=w$, we set $T^w=T$. We consider $T^x$ as a rooted tree with root $x$ and write $a(T,w;x)$ for the number of inequivalent {distinguishing sets} of $(T^x,x)$. {If the rooted tree $(T,w)$ is clear from the context, we write $a(x)$ instead of $a(T,w; x)$.} Let $y$ be the parent of $x$  and $x'$.  Then $x$ and $x'$ are twins if and only if $T^x$ and $T^{x'}$ are isomorphic. By  slight abuse of notation we also call  $\mu(x)$ the multiplicity $\mu(T^x)$ of $T^x$.

We also consider a slightly generalized version of $T^x$. Let $F$ be a set of edges, say the edge set of a one-sided infinite path, or some other subgraph of $T$. Then $T^x_F$ denotes the component of $T-F$ containing $x$. We write $T^x$ for short if the set $F$ follows  from the context. Such a $T^x$ is a \emph{lobe of $T$} with respect to $F$.

Recall that $\Delta(G)$ is the least upper bound on the degrees of the vertices of a graph $G$. If $G$ is finite, then $\Delta(G)$ is the maximum degree, and  if $\Delta(G)$ is infinite and connected, then $\Delta(G) = |G|$.

We shall frequently refer to the following  theorem of Polat and Sabidussi \cite{posa-1991}. %For the sake of completeness we include its short proof.

\begin{theorem}[{\cite[Theorem 2.3]{posa-1991}}] \label{thm:basic} The asymmetrizing number of a rooted tree $(T, w)$ is
\begin{equation}\label{eq:1}
a(T,w) = 2\prod_{{x\in R_w}}{a(x) \choose \mu(x)},
\end{equation}
where ${a \choose \mu}$ denotes the usual binomial coefficient if  $a$ and $\mu$ are finite. If $a$ is infinite, then ${a \choose \mu}=a^\mu$  if $\mu\le a$, and ${a \choose \mu}=0$ if $a < \mu$.
\end{theorem}

\begin{proof}
Let $x\in R_w$, where $R_w$ is a set of representatives for the similarity classes $\bar{y}$ of the children $y$ of the root $w$,  and $v$ be a twin of $x$. By
definition,  the number of  inequivalent distinguishing sets of the
rooted tree $T^v$ is $a(x)$, and the number of  inequivalent distinguishing sets of  $\bigcup_{v \in \bar{x}}T^v$ in $(T,w)$ clearly is   ${a(x) \choose \mu(x)}$.

Then $\prod_{{x\in R_w}}{a(x) \choose \mu(x)}$ is the number of  inequivalent  distinguishing sets of $X = \bigcup_{v\in N(w)}T^v$ in $(T,w)$, where $N(w)$ is the set of all neighbors of $w$.
The observation that for each $S\subseteq V(X)$ that distinguishes  $X$, both $S$ and $\{w\}\cup S$ distinguish $(T,w)$, completes the proof.
\end{proof}

An immediate consequence of Theorem~\ref{thm:basic} is the following corollary.%Theorem \ref{thm:basic}.

\begin{corollary}[Polat, Sabidussi {\cite[page 282]{posa-1991}}]\label{cor:ns}  A rooted tree $(T,w)$ is $2$-distinguishable if and only if $\mu(x)\le a(x)$ for all $x\in R_w$.
\end{corollary}}

We end this section with a  lemma that  holds for all  trees.

\begin{lemma}[Polat and Sabidussi \cite{posa-1991}, Proposition 2.5 (ii)]\label{le:uv} Let  $T$ be a tree with a central edge $uv$,  and $T^u$ be the component of $T- uv$ that contains $u$.  If each non-identity automorphism of $T$ swaps $u$ and $v$, then
\begin{equation}\label{eq:2}a(T) = {{a(T^u)}\choose{2}}.
\end{equation}
\end{lemma}
\begin{proof} Subdivide the edge $uv$ by vertex $w$,  apply Theorem \ref{thm:basic} to the new tree $(T,w)$, and remove $w$ again.
\end{proof}

\section{Trees with finite motion}

By  K\H{o}nig's Infinity Lemma each infinite  tree,  all of whose vertices have finite  degrees, has a one-sided infinite path. Such paths are called \emph{rays}. An immediate consequence of the lemma  is that infinite rayless trees have vertices of infinite degrees. Clearly, finite trees are also  rayless. All rayless trees, finite or infinite, have a center consisting of a single vertex or a pair of adjacent vertices that is invariant under all automorphisms of $T$. For finite trees this is well known, for infinite trees this was shown by Polat and Sabidussi  in \cite{posa-1994}, but also immediately follows from a result of Schmidt in \cite{schm-1983}. Note that the center is uniquely defined also for asymmetric   trees.

Graphs that contain rays, but no double rays, are called \emph{one-ended}. %We shall see that rayless and one-ended trees are often easier to distinguish than trees with double rays.
Each one ended tree has at least one vertex of degree 1, otherwise it would contain a double ray. It is also easy to see that each vertex of a tree with rays, is the origin of a ray. Hence, each one-ended tree contains a ray whose origin has degree 1. The subrays of a ray are called \emph{tails}, and any two of them differ at most by a finite initial segment.

For one-ended trees we need the following result of Carmesin, Lehner and M{\"o}ller \cite{carle-2019}.

\begin{lemma}\label{le:fix}{\rm (\cite{carle-2019}, Lemma 4.1)} Let $T$ be a one-ended tree and $R$ a ray of $T$. Then each automorphism of $T$ fixes some  tail of $R$ pointwise.
\end{lemma}

Recall that motion was not defined for asymmetric graphs $G$. For this case, we extend the definition by the convention that $m(G) > \alpha$ for any asymmetric graph, and any cardinal $\alpha$. In particular, a graph $G$ with fewer than $\alpha$ vertices has,  by this convention, motion $m(G)\geq \alpha$ if and only if it is asymmetric.

 Recall that the motion $m(T, w)$ of $(T,w)$  cannot be smaller than the motion $m$  of $T$, because  $\Aut(T,w) \subseteq \Aut(T)$. This means, if the degrees of $T$ are bounded by $2^m$, then they are also bounded by $2^{m(T,w)}$.

\begin{lemma}\label{le:help} Let $(T, w)$ be a finite rooted tree, where $T$ has  motion $m$, and where $\Delta(T)\le 2^{m/2}$. Then $(T,w)$ is $2$-distinguishable. Furthermore, if $\deg(w)< 2^{m/2}$, then $a(T,w) \ge 2\cdot 2^{m/2}$, and if $(T,w)$ is asymmetric, then $a(T,w)=2^{|T|}$ .
\end{lemma}

\begin{proof} If $(T,w)$ is asymmetric the lemma is true. Otherwise,  suppose that the lemma is false,  and that $(T,w)$ is a  minimal counterexample. By the minimality of $(T,w)$,  the lemma must be true for all  $T^x$, where $x$ is a child of $w$, and all $T^x$ must be pairwise isomorphic and 2-distinguishable.

 If $\mu(x) =1$, then $(T,w)$ is 2-distinguishable, because $T^x$ is. Hence, $\mu(x)\ge 2$, and  $|T^x| \ge m/2$, otherwise the motion of $(T,w)$ would be smaller than $m$.
  $T^x$ cannot be asymmetric, otherwise  $a(T^x) \ge  2^{m/2}> \mu(x)$, and $(T,w)$ is 2-distinguishable by Theorem~\ref{thm:basic}.

Hence  $T^x$, for which the lemma  holds, is not asymmetric, has motion at least $m$, $\Delta(T^x)\le 2^{m/2}$, and the degree of $x$ in $T^x$ is less than $2^{m/2}$.  We infer that $a(x) = a(T^x) \ge  2\cdot 2^{m/2} >{\rm deg}(w)= \tau(x)$.
If ${\rm deg}(w) = 2^{m/2}$, then
$$a(T, w) = 2{a(x) \choose \tau(x)}\ge 2{2\cdot 2^{m /2} \choose 2^{m /2}}\ge 2,$$
and if ${\tau(x) = \rm deg}(w)  < 2^{m/2}$, then
$$a(T, w) = 2{a(x) \choose \tau(x)}\ge 2{2\cdot 2^{m/2}\choose \tau(x)} \ge 2\cdot 2^{m/2},$$
because $\tau(x) < 2^{m/2}$. This completes the proof.
\end{proof}

%The following theorem strengthens and extends the result of Imrich and Tucker \cite[Theorem 1]{imtu-2020}.

\begin{theorem}\label{thm:fm} Let $T$ be a finite or infinite tree with finite motion $m$ and maximum degree $\Delta(T) \le  2^{m/2}$. Then:
\begin{enumerate}
\item [\rm (i)] $T$ is $2$-distinguishable.
\item [\rm (ii)]  $a(T) = 2^{|T|}$ if $T$ is infinite.
\end{enumerate}
\end{theorem}
\noindent Item (i) improves a bound of Imrich and Tucker \cite[Theorem 1]{imtu-2020}, and  item (ii) is new.

For finite trees we have no analogue to (ii), because it is possible that $a(T)=2$ for finite trees $T$ with arbitrarily large motion $m$, even when $\Delta(T) = 2^{m/2}$.
Just consider a tree $(T,w)$ with $2^{m/2}$ lobes, all of which are paths of order $m/2$.
 Then $a(T) =2$ by Theorem \ref{thm:basic}.

\begin{proof} As $T$ has bounded finite degree, and because each infinite  rayless tree has a vertex of infinite degree, $T$ must be finite if it is rayless. We thus consider the cases when $T$ is finite,   infinite with a ray but without a double ray, and infinite with double rays.

{\sc Case 1.}  $T$ is finite. Suppose that  its center consists of a single vertex $w$. Then $a(T) = a(T, w)$, and $T$ is 2-distinguishable by Lemma \ref{le:help}.

If the center of $T$ consists of an edge $uv$, then we define $T^u$ as in Lemma \ref{le:uv}. By Lemma \ref{le:help},  $T^u$ is 2-distinguishable and $T$ is 2-distinguishable by Lemma \ref{le:uv}.

{\sc Case 2.} $T$ contains a ray, but no double ray. Let $v_0$ be a vertex of degree 1 in $T$ and $R$ the ray with origin $v_0$. Then all lobes $T^x$, with respect to $R$, are rayless, because $T$  contains no double ray.  By K\H{o}nig's Infinity Lemma they are finite. Hence, the $T^x$ are finite,  $|T| = \aleph_0$, and we have to show that $a(T) = 2^{\aleph_0}$

Let $R = v_0v_1v_2\ldots $, and $c$ an arbitrary 2-coloring of $R$. The number of these colorings is $2^{\aleph_0}$, and each of them is a distinguishing coloring of $R$. For the proof of Case 2 it thus suffices to show that  each $c$ can be extended to a
2-distinguishing coloring of $T$.

 For $0\le i < \infty,$  let $A_i$ consist of the automorphisms of $T$ that preserve $c$, and fix the tail $R_{v_i}$ of $R$ with origin $v_i$. Clearly, the $A_i$ are subgroups of $\Aut(T)$, $A_0$ consists of the automorphisms of $T$ that preserve $c$, and $A_0\supseteq A_1 \supseteq \cdots $.

For $0 < i$ there are at least two  lobes of $v_i$ in $(T,{v_i})$, namely the lobe containing $v_{i-1}$,  denoted $L_i$, and the one containing $v_{i+1}$ denoted $M_i$. Furthermore, let
 $T_i'$ be the subtree of $(T, v_i)$, whose lobes are the lobes of $(T, {v_i}) -M_i$ that are isomorphic to $L_i$, and  $T''_i$ the subtree of $(T, {v_i}) -M_i$ with the remaining set of lobes. The motion of these trees cannot be smaller than $m$, and the degree of their roots is $< 2^{m/2}$. If $T''_i$  is not asymmetric, then it has $2\cdot2^{m/2}$ distinguishing 2-colorings by Lemma \ref{le:help}, and thus at least one, say $c''_i$, where   $c''_i(v_i) = c(v_i)$. If it is asymmetric, then any coloring $c''_i$ with  $c''_i(v_i) = c(v_i)$ will do. Note that $T''_i$ may consist only of $v_i$.

We now set $c_0= c$, and, for $i\ge 1$, we iteratively define colorings $c_i$ of $(T, {v_i})-M_i$ that extend $c_{i-1}$, and break all elements of $A_i$ that are not the identity on $(T,{v_i}) - M_i$.
Let $c_{i-1}$ be already defined. It colors $L_i \cup R$, and we have to extend it to $((T, {v_i})-M_i) \cup R$.

If $|L_i|< m/2$, then $L_i$ is asymmetric. Because the motion of $T$ is $m$, $T'_i$ has only one lobe, namely $L_i$.  We set $c_i = c_{i-1}\cup c''_i$. Clearly $c_i$  breaks all elements of $A_i$ that are not the identity on $(T,{v_i}) - M_i$.

If $|L_i|\ge m/2$, and if $L_i$ is asymmetric, then $T'_i$
can have up to $2^{m/2} -1$ lobes. As each lobe has at least $2^{m/2}$ inequivalent colorings, we can extend the coloring of $L_i$ to a distinguishing 2-coloring $c'_i$ of $T'_i$,  where $c'_i(v_i) = c(v_i)$.

If $|L_i|\ge m/2$, and if $L_i$ is not asymmetric, then any rooted tree isomorphic to $(L_i, v_{i-1})$ has at least $2\cdot 2^{m/2}$ distinguishing 2-colorings. So we have enough distinguishing 2-colorings for the  lobes in $T'_i$ that are different from $L_i$, to obtain a distinguishing 2-coloring $c'_i$ of $T'_i$ that extends $c_{i-1}$, and thus also $c$.

In either case,  $c_i = c'_i\cup c''_i$ breaks all elements of $A_i$ that are not the identity on $T_{v_i} - M_i$.

We claim that $c_\infty = c_0\cup c_1 \cup \cdots $ is  a distinguishing 2-coloring of $T$. If not, then there must exist a automorphism $\alpha$ and a vertex $x$ such that
$c_\infty(x) \neq c_\infty(\alpha(x))$. Let $v_{j_x}$ be the vertex on $R$ of closest distance to $x$, and $\alpha\in A_k$. Because $A_k \supseteq A_i$ for $k<i$, we can assume that $j_x<i$. Then $v_{j_x}$, and thus also $x$ is in $L_i$. But then $\alpha \in A_i$ cannot move $x$, because $c_i$ breaks all elements of $A_i$ that are not the identity on $(T, {v_i}) - M_i$.

%Note that this is even possible, when $m=2$. Then $\Delta(T)\le 2^{m/2} = 2$, and $|L_i|=1$. Then $L_i$ can only consist of $v_0$, $i =1$, and $\mathcal{L}'_i$ can consist of at  most two lobes, one being $v_0$ and the other a single vertex $u$. If such a $u$ exists we set   $c'(u) \neq c(v_0)$.

{\sc Case 3.} $T$ contains a double ray.  Then the subtree $T_*$ of $T$  that consists of all double rays of $T$ is invariant under $\Aut(T)$. As any non-identity automorphism of $T_*$ moves infinitely many vertices,
the restriction of $\Aut(T)$ to $T_*$ is the identity, and thus each automorphism of $T$ stabilizes all  lobes $T^x$ of $T$ with respect to $T_*$. If $|T^x|<m$, then $T^x$ is asymmetric, otherwise it is 2-distinguishable by Lemma \ref{le:help}.

Clearly $T_*$ is infinite, but can be of arbitrary cardinality $\kappa$. Hence the contribution of the lobes to the order of $T$ can be at most $\aleph_0\cdot |T_*| = |T_*|$. Because the vertices of $T_*$ can be arbitrarily colored, we infer that $a(T) = 2^{|T|}$.
\end{proof}

 As $T$ fulfills the assumptions of Theorem \ref{thm:fm}, $a(T) = 2^{|T|}$ for infinite $T$, but $a(T) < 2^{|T|}$ when $T$ is finite and not asymmetric. Hence, the question arises, whether $a(T) \rightarrow 2^{|T|}$ if $|T| \rightarrow \infty$ for finite $T$ with fixed motion $m$, degrees bounded by $2^{m/2}$, and at least one  non-identity automorphism.   Interestingly, this is not the case. In fact, we have the following lemma.

 \begin{lemma}\label{le:cameron} Let $G$ be a finite or infinite connected $2$-distinguishable graph. Then
\begin{equation*}\label{eq:cameron}
|\Aut(G)|\cdot a(G) \le 2^{|G|}.
\end{equation*}
\end{lemma}

\begin{proof} The claim trivially holds if $G$ is an asymmetric graph since then $|\Aut(G)|=1$ and $a(G)=2^{|G|}$. Otherwise, the proof  follows from a remark in \cite{canesa-1984} that the automorphism group of any 2-distinguishable  finite graph $G$ acts regularly on the orbit of each distinguishing set of $G$ in the power set $2^{V(G)}$ of $V(G)$. Obviously, this is also true for infinite graphs. Recall that the orbit of a subset $S$ of $V(G)$ is the set of all  $\varphi(S)$ for $\varphi\in \Aut(G)$, and that an action of a permutation group $A$ on a set $\Omega$ is regular if its action is transitive and fixed point free. By definition, the action of $\Aut(G)$ on the orbit of any $S\subset V(G)$ is transitive, and if $S$ is distinguishing, then each element of the orbit of $S$ is distinguishing, and thus it can be fixed only by the identity automorphism. As each distinguishing set has $|\Aut(G)|$ equivalent sets, we infer that $a(G)|\Aut(G)| \le 2^{|G|}$.
\end{proof}

Note that this result is independent of $m(G)$ and $\Delta(G)$.
%For finite trees that are not asymmetric this implies Equation (\ref{eq:cameron}).

\section{Trees with infinite motion}\label{sec:double}

 In \cite{po-1992} Polat defines a tree $T$ to be \emph{a-maximum} (resp.~\emph{s-maximum}) if for any
root $w$ and any neighbor $x$ of $w$, $a(T^x)= 2^{|T^x|}$ (resp.~$s(T^x)= 2^{|T^x|}$, where $s(T)$ denotes the number of inequivalent subsets of a tree $T$). He  then proves \cite[Theorem 3.3]{po-1992} about $a$-maximum and s-maximum trees, of which the part about  $a$-maximum trees reads as follows:

\begin{theorem}\label{thm:po1} A tree $T$ is $a$-maximum  if and only if, for any
vertex $w$ and any neighbor $x$ of $w$, if $T^x$ is finite then $\mu(x) = 1$, and if $T^x$  is infinite,
then $\mu(x) = 2^{|T^x|}$.
\end{theorem}

Polat also shows that $a(T) = 2^{|T|}$ for $a$-maximum  trees if and only if they are infinite or have a fixed point \cite[Theorem 3.16]{po-1992}. Because trees with infinite motion $m$ and $\Delta(T)\le 2^m$ satisfy the assumptions of Theorem \ref{thm:po1}, this clearly implies the validity of %Theorem \ref{thm:no doubleray} and
the following theorem:

\begin{theorem}\label{thm:main}
   Let $T$ be a tree of infinite motion $m$ and $\Delta(T)\le 2^m$.  Then $a(T)=2^{|T|}$.
\end{theorem}

We conclude this  section with a conjecture:

\begin{conjecture} A tree $T$ is $2$-distinguishable  if and only if, for any
vertex $w$ and any neighbor $x$ of $w$,  $\mu(x) \le a(T^x)$ if $T^x$ is finite, and  $\mu(x) = 2^{|T^x|}$ if $T^x$ is infinite.
\end{conjecture}

%The conditions of the conjecture include the asymmetrizing number of finite rooted trees, for which there are no general theorems. Nonetheless, it can be easily computed by a simple algorithm of Polat and Sabidussi \cite{posa-1991}. The algorithm is polynomial in the size of the tree.  Most likely one can adapt it to a linear variant. Because finite trees have a unique center that can be computed in linear time, this would mean that the asymmetrizing number $a(T)$ of a finite tree can be computed in linear time, and hence whether  finite trees are 2-distinguishable.

%\begin{problem} Let $T$ be a finite tree. What is the complexity of computing its distinguishing number, that is, the smallest $d$ for which $T$ is d-distinguishable? \end{problem}

\section{Tree-like graphs}\label{sec:treelike}
We now apply  Theorem \ref{thm:main} to \emph{tree-like graphs}. Imrich, Klav\v{z}ar and Trofimov defined them in \cite{imkltr-07} as rooted graphs $(G,w)$, in which each vertex $y$ has a neighbor $x$ such that $y$ is on all shortest paths from $x$ to $w$. In other words, tree-like graphs are rooted graphs $(G,w)$ in which each vertex has a child of which it is the only parent.
Note  that $w$ is the only vertex of $(G,w)$ that may have degree 1.

By \cite[Theorem 4.2]{imkltr-07},
each tree-like graph $G$ with {$\Delta(G)\le 2^{\aleph_0}$} is 2-distinguishable as an unrooted graph. We  show that $a(G) = 2^{|G|}.$

\begin{lemma}\label{doubleray}
Let $T$ be a  tree of infinite motion $m$ and  $\Delta(T)\le 2^{m}$. If each vertex of $T$  lies on a double ray, then  $T$  has $2^{|T|}$ distinguishing sets $S$ in which each vertex of $S$ is adjacent to a vertex of $V(T)\setminus S$,  and  for any $w\in V(T)$ there are $2^{|T|}$ distinguishing sets $S_w$  of $T$ where $w$ is the only vertex of $S_w$ with no neighbor in $V(T)\setminus S_w$.
\end{lemma}

\begin{proof}
Form a new tree $T'$ from $T$ by choosing an arbitrary vertex $w\in V(T)$ and by subsequently contracting all edges $ab$ to single vertices if $d_T(w,a)$ is even and $d_T(w,b) = d_T(w,a) +1$.

$T'$ has motion $m$, $\Delta(T')\le 2^m$, and $|T'| = |T|$. By  Theorem \ref{thm:main}, $a(T') = 2^{|T'|} = 2^{|T|}$. From $a(T',w)\ge a(T')$ and $a(T,w) \ge a(T)$ we also infer that $a(T',w) = a(T,w) = 2^{|T|}$.

Let $\varphi\in \Aut(T,w)$ and let $\varphi'$ denote its restriction to $V(T')$. Then $\varphi'\in \Aut(T',w)$, and $\varphi$ is uniquely determined by its action on $(T',w)$, because each vertex of $(T',w)$ different from $w$ has only one parent. This implies that a set $S\subseteq V(T)$ distinguishes $(T,w)$ if  $S' = S\cap V(T')$ distinguishes $(T',w)$.

Given a distinguishing set $S'$ of $(T',w)$ we extend it in two ways to a distinguishing set $S$ of $(T,w)$. The first is to set
$S = S'$. Clearly this implies that each vertex of $S$ has a neighbor in $V(T)\setminus S$ and, because $a(T',w) = 2^{|T|}$, the number of these  sets is $2^{|T|}$.

The second way is to form a set $S$ by adding all children of $w$ to $S'$,  unless there is a child $u$ of $w$ all of whose children are in $S'$. Then we only add $u$. Clearly we obtain $2^{|T|}$ distinguishing sets in this way, and either $w$ or $u$ are the only vertices in $S$  all of whose neighbors are in $S$.
\end{proof}

\begin{theorem} \label{thm:aleph}
Let  $G$ be a tree-like graph  with $\Delta(G) \le 2^{\aleph_0}$. Then $a(G) =2^{|G|}$.
\end{theorem}
\begin{proof}
Let  $G$ be a tree-like graph $(G,w)$. Then the set of  edges $yx$ where $y$ is on  all shortest  $x,w$-paths are the edges of a spanning subgraph, say $F$. Clearly $F$ is a forest with no finite components. Each component $U$ has motion at most $2^{\aleph_0}$, unless it is asymmetric. Because $\Delta(U)\le \Delta(G) \le 2^{\aleph_0}$ we infer by Theorem \ref{thm:main} that $a(U) = 2^{|U|}\ge 2^{\aleph_0}$.

Observe that, by the definition of $F$, any automorphism $\varphi$  of $G$ that fixes $w$ preserves the components of $F$. This is  the case when $w$ has degree 1.

We now distinguish the components of $F$ under the following restrictions: Let $U_w$ be the component of $F$ that contains $w$. If $w$ has degree 1 we admit any distinguishing set $S$ of $U_w$, otherwise we only admit distinguishing sets $S$ in which  there is only vertex that has no neighbors in $V(U_w)\setminus S$.
For all other components $U$ of $F$ we only admit distinguishing sets $S$ where each vertex has a neighbor in $V(U)\setminus S$.

Because $F$ has at most $2^{\aleph_0}$ components we can distinguish them pairwise inequivalently with admitted distinguishing sets $S_U$. That is, we distinguish by sets $S_U$ that obey the above restrictions, and if  $U,U'$ are isomorphic components of $F$, then there is no isomorphism from $U$ to $U'$ that maps $S_U$ into $S_{U'}$.

Let $S$ be the union of all $S_U$ in  such a selection. If $w$ has degree 1, then each $\varphi\in \Aut(G)$ fixes $w$. Otherwise each $\varphi\in \Aut(G)$ that preserves $S$ fixes $w$ because $w$ is the only vertex of $S$ that has no neighbors in $V(G)\setminus S$. Hence, in either case each $\varphi\in \Aut(G)$ that preserves $S$ also preserves $F$. As each $S_U$ distinguishes $U$, the set $S$ distinguishes $F$, and thus also $G$, because $F$ is a spanning subgraph of $G$.

Let $\alpha$ be the cardinality of the set of components of $F$. By transfinite induction with respect to the well ordering of the components of $F$ induced by $\alpha$, it is easily seen that there are $2^{|G|}$  distinguishing sets $S$ of $F$.
\end{proof}

\begin{problem} Let $m$ be an uncountable cardinal and $G$ be a tree-like graph with $m(G)=m$ and $\Delta(G)\le 2^m$. Is $G$ $2$-distinguishable?
\end{problem}

\section{Rayless and one-ended trees with infinite motion}

We now prove, without invoking Theorem  \ref{thm:po1}, % respectively \cite[Theorem 3.3]{po-1992} of Polat,
that rayless and one-ended trees of infinite motion satisfy Theorem \ref{thm:main}.

Our proof uses the rank of rayless graphs. It  was introduced by Schmidt in \cite{schm-1983},   but we define it as Halin in \cite{ha-1998}.

For all ordinals $\lambda$ define the following class $A(\lambda)$ of graphs by transfinite induction:
\begin{enumerate}
\item $A(0)$ consists of all finite graphs.
\item If $\lambda > 0$ and $A(\mu)$ has been defined for each $\mu < \lambda$, then a graph $G$ is in $A(\lambda)$ if there exists a finite set $S$ of vertices such that each component of $G - S$ is in an $A(\mu)$ with $\mu < \lambda$.
\end{enumerate}
Let $A$ be the union of the classes $A(\lambda)$. Then for every $G\in A$ there exists a smallest $\lambda$ with $G\in A(\lambda)$; this $\lambda$ is called the \emph{rank} of $G$ and denoted by $\rho(G)$.

One can then show, see Halin \cite{ha-1998},  that the class of graphs that is assigned a rank coincides with the class of rayless graphs,  and that there is a canonical choice for the set $S$ in the definition above, by choosing $S$  minimally among all sets that work. This minimal set is called the \emph{core} of $G$. Furthermore, the core is unique, setwise fixed by each automorphism of $G$,   and  if $S' \subseteq V(G)$ contains the core, then every component of $G -S'$ has rank less than $\rho$.
The proofs are neither difficult nor long, but subtle.

Note that the existence of a unique finite core implies that each rayless tree $T$ has  a center. Just consider the minimal subtree $T_S$ of $T$  that contains the core $S$.  This tree $T_S$ is finite,  and is preserved by all automorphisms of $T$, and thus so is its center.
\bigskip

\noindent{\it Direct proof of Theorem \ref{thm:main} for rayless and one-ended trees}.}
For rayless trees we use transfinite induction on the rank of $T$. Trees of rank 0 are finite, so they  have infinite motion only if they are asymmetric, and therefore $a(T)= 2^{|T|}$ {for every tree of rank 0 {that} satisfies the conditions of the theorem}.

For the induction step, let {$\rho>0$}  be any ordinal, {let} $T$ be a tree of rank $\rho$, and assume that
the statement of the theorem holds for any tree with rank $\sigma < \rho$.
Let $T_S$ be the minimal subtree containing the core of $T$. {If $T$ has a central vertex, then let $w$ be this central vertex. Otherwise, let $w$ be one endpoint of the central edge. Consider the rooted tree $(T,w)$.} {Clearly, the motion of $(T,w)$ is at least $m$.}

We claim that $a(x) = 2^{|T^x|}$ for every vertex $x$ of $(T,w)$. For vertices not contained in $T_S$ this is true by the induction hypothesis.  If there  are vertices that do} not satisfy the claim, {let $x$ be such a vertex at maximal distance from $w$; note that the maximal distance is finite because $T_S$ is finite. All children of $x$ satisfy the claim, and by Theorem \ref{thm:basic} the claim is satisfied for $x$ as well. This yields a contradiction.

If there is a central edge  we invoke Lemma \ref{le:uv}.

This leaves the case when  $T$ is one-ended. As in the proof of Case 2 of Theorem \ref{thm:fm} we choose a ray $R = v_0v_1\cdots$, where $v_0$ has degree 1, and note that the lobes $T^{v_i}$ of $T$ with respect to $R$ are rayless, because $T$ has no double rays.

We first show that $T$ is asymmetric if $T^{v_i}$ is finite  for all $i\ge 0$. Clearly all $T^{v_i}$ are asymmetric. Let $\alpha \in \Aut(T)$. By Lemma \ref{le:fix}, $\alpha$ fixes a tail of $R$, say $R_{v_k}$. It thus fixes all $v_j$ with $j\ge k$, and  all lobes $T^{v_j}$ for $j\ge k$.
The remaining part of $T$ is a finite subtree, and thus asymmetric.  Hence, $\alpha= \id$,  $T$ is asymmetric, and $a(T) = 2^{|T|}$.

Therefore at least one $T^{v_i}$ has infinite order. In this case we color all vertices of $R$ black, and all vertices that are adjacent to a vertex of $R$ white. Let $c$ be this coloring. Then all color preserving automorphisms of $T$ fix $R$ pointwise, and each  $T^{v_i}$ setwise. We have to extend $c$ to $2^{|T|}$ distinguishing 2-colorings of $T$.

This means that we have to find a distinguishing 2-coloring for each $T^{v_i}$, in which $v_i$ is black and all its neighbors white. If $T^{v_i}$ is finite we  color all  vertices but $v_i$ white.

If $T^{v_i}$ is infinite, then it is rayless, of motion at least $m$, and maximum degree at most $2^m$. Therefore it has $2^{|T^{v_i}|}$ distinguishing colorings.
Let $I$ be the set of indices $i$, for which $T^{v_i}$ is infinite. Because the total order of the finite $T^{v_i}$ is at most countable, we infer that $|T| = \sum_{i\ge 0}|T^{v_i}|=\sum_{i\in I}|T^{v_i}|$. Thus  the number of distinguishing 2-colorings we thus obtain is
$$\prod_{i\in I}2^{|T^{v_i}|} = 2^{\sum_{i\in I}|T^{v_i}|} = 2^{|T|}.$$\qed

% We have already seen an example with large motion, but the binary tree $B$ of arbitrary height is an example of a tree of motion 2, distinguishing number 2, and arbitrarily large size.

%There also exist arbitrarily large rayless rooted trees whose distinguishing number is 2, and   rayless trees with arbitrarily large distinguishing number, see Polat and Sabidussi \cite[Remark 3.4]{posa-1991}.


\begin{thebibliography}{99}


\bibitem{ac-1996}M. O. Albertson, K. L. Collins, Symmetry breaking in graphs, Electron. J. Combin. 3 (1996) 1--17.

\bibitem{ba-1977}L.~Babai, Asymmetric trees with two prescribed degrees, Acta Math.~Hungar. 29 (1977) 193--200.

\bibitem{ba-2021}L.~Babai, Asymmetric coloring of locally finite graphs and profinite permutation groups:
Tucker’s Conjecture confirmed, J. Algebra 607 (2022), 64--106.


\bibitem{canesa-1984}P.~J.~Cameron, P.~M.~Neumann, J.~Saxl,
On groups with no regular orbits on the set of subsets.
Arch. Math. 43  (1984), 295--296.

\bibitem{carle-2019}
 {J.~Carmesin, F.~Lehner, R.~M{\"o}ller},
 {On tree-decompositions of one-ended graphs}, {Math. Nachr.},
 292   (2019), {524--539}.
% doi = {10.1002/mana.201800055},


\bibitem{cotu-2011} {M.~D.~E.~Conder, T.~W.~Tucker, Motion and distinguishing number two, Ars Math. Contemp. 4 (2011), 63--72.}

%\bibitem{grse-1993} {J.~Graver, B.~Servatius,  H.~Servatius},  {Combinatorial rigidity},    {Grad. Stud. Math.}   {2}, AMS, Providence, RI, 1993.





\bibitem{ha-1998} {R.~Halin, The Structure of Rayless Graphs, Abh. Math. Sem. Univ. Hambg 68 (1998), 225--253.}

%\bibitem{imkale-2014} W.~Imrich, R.~Kalinowski, F.~Lehner,  M.~Pil{\'s}niak,  {Endomorphism breaking in graphs}, Electron. J. Combin. 21 (2014),  \#P1.16.
 %URL = {www.combinatorics.org/ojs/index.php/eljc/article/view/v21i1p16}


\bibitem{imkltr-07}
W.~Imrich, S.~Klav\v{z}ar, V.~Trofimov, Distinguishing infinite graphs, {Electron.~J.~Combin.} {14} (2007), 1--12.

\bibitem{imtu-2020}
W.~Imrich, T.~W.~Tucker,  Asymmetrizing trees of maximum valence $2^{\aleph_0}$, Monatsh. Math. 192 (2020), 615--624.


\bibitem{po-1991} N.~Polat,  Asymmetrization of infinite trees,  Discrete Math. 95 (1991) 291--301.

\bibitem{po-1992}N.~Polat,  Similarity and asymmetrization of trees, Discrete Math. 109 (1992)  221–-238.

\bibitem{posa-1991}  N.~Polat, G.~Sabidussi,  Asymmetrizing sets in trees,  Discrete Math. 95 (1991) 271--289.

\bibitem{posa-1994} N.~Polat,  G.~Sabidussi, Fixed elements of infinite trees, Discrete Math. 130 (1994) 97--102.

\bibitem{rusu-1998} {A.~Russel, R.~Sundaram, A note on the asymptotics and computational complexity of graph distinguishability, Electron. J. Combin. 5(1) (1998) \#R23.}

\bibitem{schm-1983} R.~Schmidt, Ein Ordnungsbegriff f\"{u}r Graphen ohne unendliche Wege mit Anwendung auf n-fach zusammenh\"angende Graphen, Arch. Math. 40 (1983), 283--288.



\bibitem{tu-2011} T.~W.~Tucker, Distinguishing maps,   Electron. J. Combin. 18 (2011) \#P50.



%\bibitem{vopu-1965} P.~Vopenka, A.~Pultr,   Z. Hedrlin,  A rigid relation exists on any set,  Comment. Math. Univ. Carolin. 6 (1965), 149--155.

\end{thebibliography}
\end{document}